\documentclass{amsart}
\usepackage{amsmath,amssymb}
\usepackage{float}
\usepackage{color}
\usepackage{hyperref}
\usepackage{graphicx}

\usepackage{enumerate}
\numberwithin{equation}{section}

\def\eps{\varepsilon}
\def\ep{\varepsilon}

\def\N{\mathbb{N}}
\def\R{\mathbb{R}}
\def\rd{\mathbb{R}^d}
\def\H{\mathcal{H}}
\def\cG{\mathcal G}
\def\cM{\mathcal M}
\newcommand{\cL}{\mathcal{L}}
\def\Ha{\mathcal{H}}
\newcommand{\sM}{\mathcal{M}}
\renewcommand{\S}{\mathbb{S}}

\newcommand{\diam}{\operatorname{diam}}

\newcommand{\dist}{\operatorname{dist}}

\newcommand{\INt}{\operatorname{int}}
\newcommand{\conv}{\operatorname{conv}}

\newcommand{\card}{\operatorname{card}}

\newcommand{\spt}{\operatorname{spt}}
\newcommand{\crit}{\operatorname{crit}}
\newcommand{\cv}{\operatorname{cv}}
\newcommand{\reg}{\operatorname{reg}}
\newcommand{\graph}{\operatorname{graph}}
\newcommand{\bd}{\partial}
\newcommand{\exo}{{\rm exo}}
\newcommand{\Unp}{{\rm Unp}}
\newcommand{\piAB}{\pi_A^{-1}(B)}

\newcommand{\udim}{\overline{\dim}}

\theoremstyle{plain}
   \newtheorem{thm}{Theorem}[section]
   \newtheorem{lem}[thm]{Lemma}
   \newtheorem{prop}[thm]{Proposition}

\theoremstyle{definition}

   \newtheorem{rem}[thm]{Remark}
\newtheorem{ex}[thm]{Example}

\begin{document}

\title[On volume and surface area of parallel sets. II.]{On volume and surface area of parallel sets. II. Surface measures and (non-)differentiability of the volume}

\author{Jan Rataj}\curraddr{\it Jan Rataj, Charles University, Faculty of Mathematics and Physics, Soko\-lovs\-k\'a 83, 186 75 Praha 8, Czech Republic}
\author{Steffen Winter}\curraddr{\it Steffen Winter, Institute of Stochastics, Karlsruhe Institute of Technology, Englerstr. 2, D-76131 Karlsruhe, Germany}

\begin{abstract}
We prove that at differentiability points $r_0>0$ of the volume function of a compact set $A\subset \R^d$ (associating to $r$ the volume of the $r$-parallel set of $A$), the surface area measures of $r$-parallel sets of $A$ converge weakly to the surface area measure of the $r_0$-parallel set as $r\to r_0$. We further study the question which sets of parallel radii can occur as sets of non-differentiability points of the volume function of some compact set. We provide a full characterization for dimensions $d=1$ and $2$.
\end{abstract}

\thanks{Steffen Winter was supported by the German Research Foundation (DFG), grant no.\ 433621248.}





\maketitle

\section{Introduction}\label{intro}


For any nonempty compact set $A\subset\R^d$, 
the \emph{open $r$-parallel sets}
\begin{equation}\label{parallelset}
A_r:=\{x\in\rd: d(x,K)=\min_{y\in A}|x-y|< r\}, \quad r>0
\end{equation}
are frequently used for the approximation of $A$, which is recovered when $r$ tends to $0$. The volume function of $A$
$$r \mapsto V_A(r):=\lambda_d(A_r)$$
is known to be continuous and nondecreasing on $(0,\infty)$, and differentiable at all $r>0$ except countably many, see Stach\'o \cite{Stacho}. It is also known that its derivative equals the surface area $S_A(r):=\H^{d-1}(\partial A_r)$ whenever $V_A'(r)$ exists (see \cite[Corollary~2.5]{RW10}; notice that the symbol $A_r$ was used for the \emph{closed} parallel set in \cite{RW10}).

Localizations of the results for volume and surface area of parallel sets were considered in \cite{Wi19}. The main aim in that paper was to establish relations between local versions of {\em Minkowski content} and {\em S-content} of $A$ -- notions based on the local parallel volume and the local surface area of parallel sets, respectively, as $r\to 0$. They capture well the local behaviour of parallel sets as $r\to 0$.

In some applications, however, one also needs to know the behaviour of parallel sets close to some parallel radius $r_0>0$. This occurs e.g.\ when approximating a self-similar fractal by parallel sets and applying renewal theory to functionals such as volume, surface area or curvature measures, see e.g.\ Z\"ahle \cite{Za11}. Globally, the behaviour of parallel volume and parallel surface area near some $r_0>0$ are reasonably well understood. While $V_A$ is continuous, the function $r\mapsto S_A(r)$ is known to be continuous at all differentiability points of $V_A$ (see \cite{RW10}). A key result towards this was the formula for the right hand derivative $(V_A)'_+(r)$ obtained in \cite[Corollary~4.6]{HLW04}, valid at any $r>0$ and connecting $(V_A)'_+(r)$ to $S_A(r)$.

Local versions of this formula have been obtained in \cite{Wi19}, in particular, \cite[Proposition~2.10]{Wi19}. They allow to study the parallel surface area locally by restricting it to suitable sets (which need to be metrically associated with $A$). Recently, more general formulas valid for parallel sets defined with respect to other norms have been obtained in \cite[Theorem~4.3]{HS22} and \cite[Theorem~5.2]{CLV21}.

It is our aim here to complete the picture (in the Euclidean setting) by studying the continuity properties of the local parallel surface area. We find that the above mentioned global continuity result localizes. For any suitable fixed restricting set, the local surface area turns out to be continuous at differentiability points of the volume function, see Lemma~\ref{lem:conv-PiA} below. Further insights come from viewing the local parallel surface area as a measure in the restricting set. Our main result is a measure version of the continuity result. The continuity at differentiability points is also true for the surface measures in the weak sense. We denote
\begin{align*}
  S_A(r,\cdot):=\H^{d-1}(\partial A_r\cap \cdot)
\end{align*}
the \emph{parallel surface measure} of $A$ at distance $r\geq 0$.

\begin{thm}   \label{thm:main}
Let $A\subset\R^d$ be nonempty and compact and let $r_0>0$ be a differentiability point of $r\mapsto V_A(r)$. Then the measures $S_A(r,\cdot)$ converge weakly to $S_A(r_0,\cdot)$ as $r\to r_0$.
\end{thm}

Note that the weak continuity of $S_A(r,\cdot)$ at points $r_0$ which are \emph{regular values} of the distance function of $A$ follows from \cite[Theorem~6.1]{RWZ22} (where the weak continuity of all Lipschitz-Killing curvature measures is shown). The set of critical (i.e., non-regular) values can be large compared to the set of non-differentiability points, see also the discussion below.

The proof (provided in Section~\ref{S_proof}) is based on some results from \cite{Wi19}. Unfortunately, there appears to be an error in \cite[Lemma~2.9]{Wi19}. We correct the statement in Section~\ref{s_corr} and present also corrected proofs of \cite[Proposition~2.10]{Wi19} and \cite[Theorem~2.7]{Wi19}, which in their original version used the defect Lemma~2.9.

In Section~\ref{s_cons} we formulate some consequences of these results. In particular, we show that the volume function $V_A$ is differentiable at some $r>0$ if and only if the set of critical points of the distance function of $A$ lying at distance $r$ from $A$ has positive $(d-1)$-dimensional Hausdorff measure $\Ha^{d-1}$, see Theorem~\ref{Cor_3.1}. For the Euclidean setting this recovers and gives a streamlined argument for the observation in \cite[Theorem~4.3]{HS22}, that $V_A$ is non-differentiable at some $r>0$ if and only if the set $\partial A_r\setminus\Unp(A)$ has positive $\Ha^{d-1}$-measure, where $\Unp(A)$ is the set of points with a unique nearest point in $A$.

The final part (Section~\ref{s_N}) is devoted to the question which sets $N\subset (0,\infty)$ can occur as sets of non-differentiability points of $V_A$ for some compact set $A\subset\R^d$. Of course, such a set $N$ has to be countable but, clearly, there are further limitations. It turns out that the answer to this question depends heavily on the space dimension $d$. For dimension 1 we obtain the following characterization. For $A\subset\R^d$ denote by $N_A$ the set of non-differentiability points of $V_A$.

\begin{thm}\label{thm:nondiff-points-dim1}
  Let $N\subset(0,\infty)$.  Then there exists some compact set $A\subset\R$ such that
  $N_A=N$ if and only if  $\sum_{s\in N} s<\infty$.
\end{thm}

Note that the condition $\sum_{s\in N} s<\infty$ implies that $N$ is countable and bounded and moreover, that the set $N$ can only accumulate at $0$.

In dimension $2$ the situation happens to be substantially more complicated. Due to the already mentioned relation between non-differentiability of $V_A$ and the size of the set of critical points at level $r$, cf.\ Theorem~\ref{Cor_3.1}, we can apply a result for critical points from \cite{RZ20}. The conditions are, however, different; note that the set of critical values does not need to be countable and it is always closed. We use the notation $G_{1/2}(K)$ for the degree-$\frac 12$ gap sum of a compact set $K\subset\R$, which is the sum of the square roots of the lengths of all bounded intervals of the complement of $K$ (or, equivalently, the value of the geometric zeta function $\zeta_{\cL}$ at $\frac 12$ of the fractal string $\cL$ associated to $K$, cf.~\cite{La06}).

\begin{thm}  \label{thm:dim2}
\begin{enumerate}
\item[(i)] Let $\ep>0$ and $N\subset[\ep,\infty)$ be a bounded countable set. Then $N=N_A$ for some compact set $A\subset\R^2$ if and only if $\lambda_1(\overline{N})=0$ and $G_{1/2}(\overline{N})<\infty$.
\item[(ii)] Let $N\subset(0,\infty)$ be a bounded countable set. Then $N=N_A$ for some compact set $A\subset\R^2$ if and only if $\lambda_1(\overline{N})=0$ and
\begin{equation}  \label{int_bound}
\int_0^\infty G_{1/2}(\overline{N}\cap [r,\infty))\, \sqrt{r}\, dr<\infty.
\end{equation}
\end{enumerate}
\end{thm}

Note that in contrast to the case $d=1$, for $d=2$, sets $N_A$ may have accumulation points different from $0$, which may or may not be contained in $N_A$. E.g., the sets $N=\{1+n^{-2}: n\in\N\}$ and $\overline{N}$ are valid sets of non-differentiability points. However, there are more interesting examples, such as sets of endpoints of segments defining certain Cantor sets, see Example~\ref{ex:cantor}. 

The gap sum conditions in Theorem~\ref{thm:dim2} are closely connected with (upper) Minkowski contents of $\overline{N}$: 
the condition in (i) implies that $\overline{\cM}^{1/2}(\overline{N})=0$ (and thus $\udim_M \overline{N}\leq \frac 12$), whereas the condition in (ii) only implies $\overline{\cM}^{4/5}(\overline{N})=0$ (and thus $\udim_M \overline{N}\leq \frac 45$), see \cite[Theorem~3.17]{RZ20}. Here $\overline{\cM}^\alpha(X)$ and $\udim_M X$ denote the $\alpha$-dimensional upper Minkowski content and the upper Minkowski dimension, respectively, of a set $X$.

The case of dimension $d>2$ seems to be even more complicated. We do not know a necessary and sufficient condition here. We present only a sufficient condition, see Proposition~\ref{prop:NC-dimd}. A full characterization  remains as an open question.

\section{Corrigendum to \cite{Wi19}}  \label{s_corr}

There is an unfortunate error in \cite[Lemma~2.9]{Wi19}. In particular, equation \cite[(2.7)]{Wi19} in this statement is not an equality in general but only an inclusion. This affects \cite[Proposition 2.10]{Wi19}, where the second equality in equation \cite[(2.8)]{Wi19} is only a `$\leq$' in general. Equality holds in the case when $r$ is a differentiability point of $V_A$. The statement of \cite[Theorem~2.7]{Wi19} remains true as formulated but its proof has to be modified slightly. We provide here correct versions of these statements together with corrected proofs.

First we recall some notation from \cite{Wi19}. Assume throughout that $A\subset\R^d$ is compact. Let $\Unp(A)\subset\R^d$ be the set of those points which have a unique nearest point in ${A}$. Let $\pi_A:\Unp(A)\to A$ denote the \emph{metric projection} onto the set $A$\label{page:metric-proj}, which assigns to each $x\in\Unp(A)$ its unique nearest point $\pi_A(x)$ in $A$. (Note that $\piAB$ is a subset of $\Unp(A)$ for any set $B$.) Recall that the set $\exo(A):=\R^d\setminus\Unp(A)$, the \emph{exoskeleton} of $A$, consists of all points that do not have a unique nearest point in $A$. Note that $A\subset\Unp(A)$ and thus $\exo( A)\subset A^c$.

A set $X\subset\R^d$ is said to be \emph{metrically associated with} $A$, if for any point $x\in X$ there exists a point $y\in {A}$ so that
$d(x,y)=d(x,A)$ and all inner points of the line segment $[x,y]$ joining $x$ with $y$ belong to $X$ (cf.\ \cite[p.370]{Stacho}). We write $]x,y[$ for the line segment excluding the endpoints.
Observe that the parallel sets $A_r$ of $A$ are metrically associated with $A$. Moreover, any set of the form $\pi_A^{-1}(B)$ with $B\subset A$ (or $B\subset\R^d$) is metrically associated with $A$.
Combining these two constructions, it is easy to see that all sets of the form $A_r\cap \piAB$ are metrically associated with $A$.

The \emph{positive boundary} $\partial^+ Z$ of a set $Z\subset\R^d$ is the set of all boundary points $z\in\bd Z$ such that there exists a point $y\notin \overline{Z}$ with $|y-z|=d(y,Z)$.

\begin{lem}[Correction of {\cite[Lemma~2.9]{Wi19}}]
  \label{lem:bd=pos-bd}
Let $A, B \subseteq\R^d$. Then, for any $r>0$,
\begin{align} \label{eq:bd=pos-bd}
  \partial^+ (A_r)\cap\piAB\subseteq\partial (A_r)\cap\piAB.
\end{align}
In particular, $\partial^+ (A_r)\cap\Unp(A)\subseteq \partial (A_r)\cap\Unp(A)$. 
Moreover, $\partial^+ (A_r)\subseteq \Unp(A)$ and thus $\partial^+ (A_r)\subseteq\partial (A_r)\cap\Unp(A)$.
\end{lem}

\begin{proof}
  The inclusion \eqref{eq:bd=pos-bd} is obvious from the definition of the positive boundary as a subset of the boundary and the second stated inclusion is a special case of the first one. Therefore, only the inclusion $\partial^+ (A_r)\subseteq \Unp(A)$ requires a proof. We recall here the argument given in \cite{Wi19} correcting another typo: Let $z\in\partial^+ (A_r)$. Then there exists a point $y$ outside $\overline{A_r}$ such that $|y-z|=d(y,A_r)$.
   By way of contradiction, assume that $z$ does not admit a unique metric projection onto $A$, i.e., there exist at least two distinct points $a,b\in A$ such that $|z-a|=|z-b|=d(z,A)=r$. At least one of those two points, say $a$, is not on the ray from $y$ through $z$. Therefore, we have \[|y-a|<|y-z|+|z-a|=|y-z|+r.\]
   Let $z'$ be the point on the segment $[y,a]$ such that $|z'-a|=r$. Then, in particular, $z'\in \overline{A_r}$ and $|y-a|=|y-z'|+|z'-a|=|y-z'|+r$. Plugging this into the above inequality and subtracting $r$ yields $|y-z'|<|y-z|$ and thus $d(y,A_r)<|y-z|$, which is a contradiction to the choice of $y$. Hence, $z\in\Unp(A)$.
\end{proof}

For any closed set $A\subset\R^d$ and any measurable set $B\subset\R^d$, we write 
\[
V_{A,B}(r):=\lambda_d(A_r\cap \piAB), \quad r>0
\]
for the \emph{local parallel volume of $A$ relative to the set $B$}. Recall that $A_r\cap \piAB$ is metrically associated with $A$ and thus, if $V_{A,B}$ is finite for some $r>0$, then it is finite for all $r>0$ and it is a Kneser function, cf.~\cite[Lemma~2.4]{Wi19} (see \eqref{eq:def-kneser} for the definition of a Kneser function). Hence $(V_{A,B})'_+$ and $(V_{A,B})'_-$ exist at any $r>0$ and $(V_{A,B})'_+(r)\leq(V_{A,B})'_-(r)$, cf.~\cite{Stacho}.

\begin{rem} \label{rem:diff}
If $A$ is compact, then differentiability of $V_A$ at some $r>0$ implies the differentiability of $V_{A,B}$ at this $r$ for any measurable set $B$. Indeed, assuming that $(V_{A,B})'_+(r)<(V_{A,B})'_-(r)$, the general relation $(V_{A,A\setminus B})'_+(r)\leq(V_{A,A\setminus B})'_-(r)$ implies that
\begin{align*}
  (V_A)'_+(r)=(V_{A,B})'_+(r)+(V_{A,A\setminus B})'_+(r)<(V_{A,B})'_-(r)+(V_{A,A\setminus B})'_-(r)=(V_A)'_-(r).
\end{align*}
Hence $V_A$ is not differentiable at $r$, a contradiction. Here we have used that
\begin{align*}
   V_{A}(r)=V_{A,A}(r)=V_{A,\R^d}(r)=\lambda_d(A_r\cap\Unp(A))
\end{align*}
for any $r>0$ (since $\lambda_d(\exo(A))=0$, see e.g.~\cite[eq.~(2.1)]{HLW04}) and thus $(V_{A})'_+(r)=(V_{A,A})'_+(r)$ for any $r>0$ and similarly for the left derivatives.
\end{rem}

For the proof of the next statement as well as for later use we recall the notion of (relative) Minkowski content. For $s\in[0,d]$, the $s$-dimensional \emph{relative Minkowski content} of a set $A\subset\R^d$ with respect to $\Omega\subset\R^d$ is defined as
$$\sM^k(A,\Omega):=\lim_{r\to 0_+}\frac{\lambda_d(A_r\cap\Omega)}{\kappa_{d-s}r^{d-s}},$$
where $\kappa_t:=\pi^{t/2}/\Gamma(\frac t2+1)$, provided that the limit exists, cf.\ 
\cite[equation~(1.9)]{Wi19}. In particular, $\sM^s(A):=\sM^s(A,\R^d)$ is the standard $s$-dimensional Minkowski content.

\begin{prop}[Correction of {\cite[Proposition 2.10]{Wi19}}] \label{lem:pos-bd}
Let $A \subset\R^d$ be closed and let $B\subset\R^d$ be a Borel set such that $A\cap B$ is bounded. Then, for any $r>0$,
\begin{align} \label{eq:pos-bd}
  (V_{A,B})_{+}'(r)=\Ha^{d-1}(\partial^+ (A_r)\cap\piAB)\leq\Ha^{d-1}(\partial (A_r)\cap\piAB).
\end{align}
Moreover, if $A$ is compact and $V_A$ is differentiable at $r>0$, then `$\leq$' in \eqref{eq:pos-bd} can be replaced by `$=$'. Similarly, let $r>0$ and let $K\supset A\cap B$ be some ball such that $\dist(\bd K, A\cap B)> 2r$. If $V_{A\cap K}$ is differentiable at $r$, then again equality holds in \eqref{eq:pos-bd}.
\end{prop}

\begin{proof}
The proof of the equality in \eqref{eq:pos-bd} is correct as given in \cite[Prop.~2.10]{Wi19} and the inequality in \eqref{eq:pos-bd} is obvious from Lemma~\ref{lem:bd=pos-bd}. Therefore, it only remains to prove the equality cases. Assume first that $A$ is compact
and that $r>0$ is a differentiability point of $V_A$.
Then, by Remark~\ref{rem:diff} and \eqref{eq:pos-bd}, we have on the one hand
\begin{align}
   \label{eq:pos-bd-proof1} V_{A,B}'(r)&=(V_{A,B})_{+}'(r)\leq\Ha^{d-1}(\partial (A_r)\cap\piAB)
   \quad \text{ and }\\
   \label{eq:pos-bd-proof2} V_{A,A\setminus B}'(r)&=(V_{A,A\setminus B})_{+}'(r)\leq\Ha^{d-1}(\partial (A_r)\cap\pi_A^{-1}(A\setminus B)).
\end{align}
On the other hand, by the $(d-1)$-rectifiability\footnote{A subset of $\R^d$ is \emph{$k$-rectifiable} if it is a Lipschitz image of a bounded subset of $\R^k$ ($k=0,1\dots,d$).}
of $\partial A_r$ (see \cite[Proposition~2.3]{RW10}) and, thus, of any of its subsets, we have
\begin{align}
\notag\Ha^{d-1}(\partial (A_r)\cap\piAB)&+\Ha^{d-1}(\partial (A_r)\cap\pi_A^{-1}(A\setminus B))=\Ha^{d-1}(\partial (A_r)\cap\Unp(A))\\\notag &=\sM^{d-1}(\partial (A_r)\cap\Unp(A))\\
\label{eq:pos-bd-proof3}&\leq\sM^{d-1}(\partial (A_r),\Unp(A))\\
  \notag&=\sM^{d-1}(\partial (A_r),\piAB)+\sM^{d-1}(\partial (A_r),\pi_A^{-1}(A\setminus B))\\
  \notag&=V_{A,B}'(r)+V_{A,A\setminus B}'(r).
\end{align}
Here the last equality is due to \cite[Theorem~2.5]{Wi19} and the last but one equality follows from the fact that the set function $B\mapsto\sM^{d-1}(\bd (A_r),\pi^{-1}(B))$ is a measure, see \cite[Lemma~2.11]{Wi19}. Moreover, the inequality above is seen from the fact that $(\bd(A_r)\cap \Unp(A))_t\subseteq (\bd(A_r))_t\cap (\Unp(A))_t$ for any $t>0$, which implies that
\begin{align*}
  \sM^{d-1}(\partial (A_r)\cap\Unp(A))&=\lim_{t\searrow 0}\frac 1{2t}\lambda_d((\bd(A_r)\cap \Unp(A))_t)\\
  &\leq\lim_{t\searrow 0}\frac 1{2t}\lambda_d((\bd(A_r))_t\cap (\Unp(A))_t)\\
  &=\lim_{t\searrow 0}\frac 1{2t}\lambda_d((\bd(A_r))_t\cap \Unp(A))\\&=\sM^{d-1}(\partial (A_r),\Unp(A)),
\end{align*}
since $\lambda_d(\exo(A))=0$, see e.g.\ \cite[eq.~(2.1)]{HLW04}.
Combining now \eqref{eq:pos-bd-proof1},\eqref{eq:pos-bd-proof2} and \eqref{eq:pos-bd-proof3} yields the desired conclusion $V_{A,B}'(r)=\Ha^{d-1}(\partial (A_r)\cap\piAB)$. (Indeed, if $a'\leq a$, $b'\leq b$ and $a+b\leq a'+b'$ for some numbers $a,a',b,b'\in\R$, then $a'=a$.)

Finally, if $A$ is closed but not compact and $r>0$, then the assumed boundedness of $A\cap B$ allows to choose a closed ball $K$ containing $A\cap B$ such that $\dist(\bd K, A\cap B)=:s>2r$. Observe that $A_t\cap\piAB=(A\cap K)_t\cap\pi_{A\cap K}^{-1}(B)$ holds for any $0<t<s/2$ and that the set $A\cap K$ is compact. Hence if $V_{A\cap K}$ is differentiable at $r$, then one can prove the relation using the compact set $A\cap K$ instead of $A$.
\end{proof}

We recall \cite[Theorem 2.7]{Wi19} and provide a slightly corrected proof.

\begin{thm} (see \cite[Theorem 2.7]{Wi19}) \label{thm:loc-diff-points}
 Let $A \subset\R^d$ be closed and $B \subset\R^d$ a Borel set such that $A\cap B$ is bounded.
  Then, for all $r>0$ up to a countable set of exceptions, the function $V_{A,B}$ is differentiable at $r$ with
  \begin{align} \label{eq:local-deriv-vol2}
  V'_{A,B}(r)&=\sM^{d-1}(\bd A_r, \piAB )=\Ha^{d-1} (\bd A_r\cap \piAB)
  =\sM^{d-1}(\bd A_r\cap \piAB ).
\end{align}
\end{thm}

\begin{proof}[Proof of Theorem~\ref{thm:loc-diff-points}]
If $A$ is compact, then $V_A$ is differentiable for all $r>0$ except countably many. Let $r>0$ be a differentiability point of $V_A$. Then, by Remark~\ref{rem:diff}, $V_{A,B}$ is also differentiable at $r$. The first equality in \eqref{eq:local-deriv-vol2} is obvious from \cite[Theorem~2.5]{Wi19} and the second equality follows from Proposition~\ref{lem:pos-bd}. Moreover, the third equality follows from the fact (proved e.g.\ in \cite[Proposition 2.3]{RW10}) that $\bd A_r$ (and thus any of its subsets) is $(d-1)$-rectifiable and the equality of Minkowski content and Hausdorff measure for such sets, see e.g.~\cite[Section 3.2.39]{Fe69}.

If $A$ is closed but not bounded, then the third equality in \eqref{eq:local-deriv-vol2} is still true (for all $r>0$) by the same argument as above. For a proof of the other two inequalities let $t>0$ be arbitrary. The assumed boundedness  of $A\cap B$ allows to choose a closed ball $K$ such that $A\cap B\subset K$ and $\dist(\bd K,A\cap B)>2t$. Obviously $A\cap K$ is compact and $V_{A\cap K}$ is differentiable for all $r>0$ except countably many. Let $r<t$ be such a differentiability point. Observe that, by the choice of $K$, we have $V_{A\cap K, B}(s)=V_{A,B}(s)$ for all $0<s<t$, which implies in particular that
$(V_{A,B})'(r)$ exists if and only if $(V_{A\cap K,B})'(r)$ exists and both coincide in this case. Note that, by Remark~\ref{rem:diff}, $V_{A\cap K, B}$ is differentiable at $r$ and thus also $(V_{A,B})'(r)$ exists. Therefore, the first equality in \eqref{eq:local-deriv-vol2} follows now from \cite[Theorem~2.5]{Wi19}. Moreover, the hypothesis of the last assertion in Proposition~\ref{lem:pos-bd} is satisfied and we get $(V_{A,B})'(r)=\Ha^{d-1}(\partial (A_r)\cap\piAB)$ for this $r$, proving the second equality in \eqref{eq:local-deriv-vol2}.
Together this shows equation \eqref{eq:local-deriv-vol2} for all $r\in(0,t)$ except countably many. Since $t>0$ was arbitrary, this completes the proof for unbounded sets $A$.
\end{proof}

\section{Characterizations of differentiability points}  \label{s_cons}

Given a compact set $A\subset\R^d$, we denote by $\crit(A)$ the set of all critical points of the distance function $d(\cdot,A)$, and by $\reg(A):=\R^d\setminus\crit(A)$ the set of regular points of $d(\cdot, A)$. Note that $x\in\crit(A)$ if and only if $x\in\conv\Sigma_A(x)$, where $\Sigma_A(x):=\{a\in A:\, \|a-x\|=d(x,A)\}$ and $\conv$ denotes the convex hull (see \cite[Lemma~4.2]{Fu85}). By $\cv(A):=\{d(x,A):\, x\in\crit(A)\}$ we denote the set of critical values of $A$. The sets $\crit(A)\subset\R^d$ and $\cv(A)\subset[0,\infty)$ are closed. We will use the following lemma.

\begin{lem}
For any $A\subset\R^d$ compact and $r>0$ we have
$$\partial A_r\cap\partial^+A_r\subset\partial A_r\cap\Unp(A)\subset\partial A_r\cap\reg(A).$$
Further, for any $r>0$, both inclusions are equalities up to an $\H^{d-1}$-null set, i.e.,
\begin{align}
\label{eq:lem5.2-1} \H^{d-1}&(\partial A_r\cap(\Unp(A)\setminus\partial^+A_r))=0,\\
\label{eq:lem5.2-2} \H^{d-1}&(\partial A_r\cap(\reg(A)\setminus\Unp(A)))=0.
\end{align}
\end{lem}

\begin{proof}
The first set inclusion is already obtained in Lemma~\ref{lem:bd=pos-bd} and the second one is obvious (a point outside of $A$ having a unique nearest neighbour in $A$ must clearly be regular). For the equations \eqref{eq:lem5.2-1} and \eqref{eq:lem5.2-2}, we use the fact that for a regular point $x\in\partial A_r\cap\reg(A)$, the boundary $\partial A_r$ is a Lipschitz surface in some neighbourhood of $x$ (see \cite[Theorem~3.3]{Fu85}). Hence, by Rademacher's theorem, the outer normal $n(y)$ is differentiable $\H^{d-1}$-a.e.\ on this surface, and any point $y\in\partial A_r$ where $n(y)$ is differentiable belongs to $\partial^+A_r$. 
(Indeed, if a function $h$ of $d-1$ variables has differentiable normalized gradient $n(u)$ at a point $u$ then $h$ has second derivative at $u$ and standard analysis yields that for $x=(u,h(u))\in \graph h$ and some $t>0$, the ball $B(x+tn(u),t)$ supports $\graph h$, i.e., it contains no points of $\graph h$ in its interior, and it easily follows that $\pi_A(x+sn(u))=\pi_{\graph h}(x+sn(u))=x$ if $s>0$ is small enough.)  Observe that these Lipschitz surfaces cover the set $\bd A_r\cap \reg(A)$.
\end{proof}

From the above results one can easily derive the following
characterization of differentiability of the volume function.

\begin{thm}  \label{Cor_3.1}
  Let $A\subset \R^d$ be a compact set and $r>0$. Then the following assertions are equivalent:
  \begin{enumerate}[(i)]
    \item the volume function $V_A$ is differentiable at $r$;
    \item $\Ha^{d-1}(\bd A_r\setminus \bd^+ A_r) =0$;
    \item $\Ha^{d-1}(\bd A_r\setminus \Unp(A)) =0$ and $\Ha^{d-1}(\bd A_r\cap\Unp(A)\setminus \bd^+ A_r) =0$;
    \item $\Ha^{d-1}(\bd A_r\setminus \Unp(A)) =0$;
		\item $\Ha^{d-1}(\partial A_r\cap\crit(A))=0$.
  \end{enumerate}
\end{thm}

\begin{rem}  \label{rem}
The equivalence of (i) and (v) implies that if $V_A$ is not differentiable at $r>0$ then $r$ must be a critical value of $d(\cdot,A)$.

The equivalence of (i) and (iv) has been obtained as a side result in \cite{HS22} using a slightly different argument and not only for the usual (Euclidean) parallel sets but also for parallel sets w.r.t.\ some other Minkowski norms, see \cite[Thm.~4.3]{HS22}. Also the other equivalences appear essentially in \cite{HS22}, see in particular \cite[Lemma 4.2 and eq.\ (60)]{HS22}. Moreover, one can deduce from \cite[Lemma 3.25]{HS22}, that $\Ha^{d-1}$-a.a.\ critical points  $x$ of $A$ have exactly two nearest points (located in opposite directions such that they lie on a segment containing $x$). Hence in (v), the set $\crit(A)$ can equivalently be replaced by the subset $\{x\in\crit(A): \card\Sigma_A(x)=2\}$. 
Yet another slightly different characterization of the differentiability of $V_A$ is obtained in \cite[Theorem 5.2]{CLV21}, see also \cite[Remark~4.4]{HS22} for the relation.
\end{rem}

\begin{proof}
  $(i) \Leftrightarrow (ii)$: If $V_A$ is differentiable at $r>0$, then, by \cite[Corollary~2.5]{RW10}, $V_A'(r)=\Ha^{d-1}(\bd^+ A_r)=\Ha^{d-1}(\bd A_r)$, which implies
$\Ha^{d-1}(\bd A_r\setminus \bd^+ A_r) =0$.

If $V_A$ is not differentiable at $r>0$, then one has $(V_A)'_+(r)<(V_A)'_-(r)$. Hence, by \cite[Corollary~4.6]{HLW04}, Stachó's Theorem (cf.\ e.g. \cite[eq.~(2.1)]{RW10}) and \cite[Corollary~2.4]{RW10}, we get
\begin{align*}
  \Ha^{d-1}(\bd^+ A_r)=(V_A)'_+(r)<\frac 12((V_A)'_+(r)+(V_A)'_-(r))=\sM^{d-1}(\bd A_r)=\Ha^{d-1}(\bd A_r).
\end{align*}
Since $\bd^+ A_r\subseteq \bd A_r$, this implies immediately $\Ha^{d-1}(\bd A_r\setminus \bd^+ A_r)>0$.

$(ii) \Leftrightarrow (iii)$: This is obvious from the disjoint decomposition $\bd A_r\setminus \bd^+ A_r=\bd A_r\setminus \bd^+ A_r\setminus \Unp(A)\cup\bd A_r\setminus \bd^+ A_r\cap\Unp(A)$ and the set inclusion $\partial^+ A_r\subset\Unp(A)$ (see Lemma~\ref{lem:bd=pos-bd}).

$(iii) \Leftrightarrow (iv)$: Follows from \eqref{eq:lem5.2-1}.

$(iv) \Leftrightarrow (v)$: Follows from \eqref{eq:lem5.2-2}.
\end{proof}

\section{Proof of the main result}  \label{S_proof}

In this section we provide a proof of Theorem~\ref{thm:main}. Recall that
a \emph{Kneser function} (of order $d$) is a continuous function $f:(0,\infty)\to(0,\infty)$ such that for all $\lambda\geq 1$,
\begin{align}
  \label{eq:def-kneser}
  f(\lambda b)-f(\lambda a)\leq\lambda^d(f(b)-f(a)),\quad 0<a\leq b.
\end{align}
Any Kneser function $f$ is differentiable at all points $r>0$ except countably many and that left and right derivatives exist at any point $r>0$ and satisfy $f'_-(r)\geq f'_+(r)$, see \cite{Stacho}. Moreover, $f'_-$ is left-continuous and $f'_+$ is right-continuous. These properties imply continuity of $f'_-$ and $f'_+$ at differentiability points of $f$:
\begin{lem} \label{lem:cont-f'}
  Let $f:(0,\infty)\to\R$ be a Kneser function and let $D\subset(0,\infty)$ be the set of differentiability points of $f$. Then for any $r_0\in D$,
  \begin{align} \label{eq:cont-f'}
    f'(r_0)=\lim_{r\to r_0} f'_+(r)=\lim_{r\to r_0} f'_-(r)=\lim_{{r\to r_0},{r\in D}} f'(r).
  \end{align}
\end{lem}

\begin{proof}
First note that the last equality is a consequence of the first two. We provide a proof of the first equality, the second one can be proved similarly. Let $r_0\in D$. Since $f'_+$ is well known to be right-continuous at any point, it suffices to show that $f'_+$ is left-continuous at $r_0$. By the left-continuity of $f'_-$, for any $\eps>0$ there exists $\delta>0$ such that $f'_-(s)\in I:=[f'(r_0)-\eps,f'(r_0)+\eps]$ whenever $r_0-\delta<s<r_0$.  For any such $s$ we have also (by the right-continuity of $f'_+$)
$$f'_+(s)=\lim_{t\to s_+}f'_+(t)=\lim_{t\to s_+,t\in D}f'_+(t)=\lim_{t\to s_+,t\in D}f'_-(t)\in I,$$
hence, also $f'_+$ is left-continuous at $r_0$.
\end{proof}

Let $N(A)$ denote the (generalized) normal bundle of $A$. (Recall that $N(A)$ consists of all pairs $(x,n)\in\partial A\times S^{d-1}$ such that $\pi_A(x+sn)=x$ for some $s>0$, see \cite[p.239]{HLW04}.) Let $\Pi_A:\Unp(A)\setminus A\to N(A)$ be defined by $x\mapsto (\pi_A(x), \frac{x-\pi_A(x)}{|x-\pi_A(x)|})$. Then for any subset $\beta\subset N(A)$, the sets $A_r\cap \Pi_A^{-1}(\beta)$ are again metrically associated with $A$ and hence, the function $r\mapsto V_{A,\beta}(r):=\lambda_d(A_r\cap\Pi_A^{-1}(\beta))$ is a Kneser function of order $d$. Observe that working with $\Pi_A$ instead of $\pi_A$ is a significant refinement, allowing e.g.\ to study one sided parallel sets of a curve in $\R^2$ or a surface in $\R^3$.

As indicated in \cite[Remark 2.1]{Wi19},  all results for preimages $\piAB$ of (measurable) sets $B$ in \cite{Wi19} extend to preimages $\Pi_A^{-1}(\beta)$ of (measurable) sets $\beta$ in the normal bundle. Note that also the set inclusion \eqref{eq:bd=pos-bd} holds with $\piAB$ replaced by $\Pi_A^{-1}(\beta)$, and similary as explained in Remark~\ref{rem:diff}, differentiability of $V_A=V_{A,A\times\S^{d-1}}$ at some $r>0$ implies the differentiability of $V_{A,\beta}$ at this $r$ for any measurable $\beta\subseteq A\times \S^{d-1}$.
For the convenience of the reader, we formulate here the corresponding generalization of (the relevant part of) Proposition~\ref{lem:pos-bd}. It will be used in the proof of the next lemma.
\begin{prop}\label{lem:pos-bd-PiA}
Let $A \subset\R^d$ be nonempty and compact and let $\beta\subset N(A)$ be some Borel set. Then, for any $r>0$,
\begin{align} \label{eq:pos-bd-PiA}
  (V_{A,\beta})_{+}'(r)=\Ha^{d-1}(\partial^+ A_r\cap\Pi^{-1}_A(\beta))\leq\Ha^{d-1}(\partial A_r\cap\Pi^{-1}_A(\beta)).
\end{align}
Moreover, if $V_A$ is differentiable at $r>0$, then `$\leq$' in \eqref{eq:pos-bd} can be replaced by `$=$'. 
\end{prop}
Recall that $S_A(r,\cdot)$ denotes
the parallel surface measure of $A$ at distance $r\geq 0$.
\begin{lem} \label{lem:conv-PiA}
  Let $A\subset\R^d$ be nonempty and compact and let $r_0>0$ be a differentiability point of $r\mapsto V_A(r)$. Then for any Borel set $\beta\subset N(A)$
  \begin{align*}
    \lim_{r\to r_0} S_A(r,\Pi_A^{-1}(\beta))=S_A(r_0,\Pi_A^{-1}(\beta)).
  \end{align*}
\end{lem}
\begin{proof}
Observe that $f:(0,\infty)\to\R, r\mapsto \lambda_d(A_r\cap \Pi_A^{-1}(\beta))$ is a Kneser function. Since $r_0$ is assumed to be a differentiability point of $V_A$, by Remark~\ref{rem:diff}, $f$ is also differentiable at $r_0$.
Applying Proposition~\ref{lem:pos-bd-PiA} and Lemma~\ref{lem:cont-f'}, we conclude
\begin{align} \label{eq:liminf-mur}
   \liminf_{r\to r_0} S_A(r,\Pi_A^{-1}(\beta))&=\liminf_{r\to r_0} \Ha^{d-1}(\bd A_r\cap\Pi_A^{-1}(\beta)) \\
   &\geq \liminf_{r\to r_0} \Ha^{d-1}(\bd^+ A_r\cap\Pi_A^{-1}(\beta))\notag \\
   &=\liminf_{r\to r_0} f'_+(r)=f'(r_0)=S_A(r_0,\Pi_A^{-1}(\beta)).\notag
\end{align}
To complete the proof, it remains to show
\begin{align} \label{eq:limsup-mur}
   s:=\limsup_{r\to r_0} S_A(r,\Pi_A^{-1}(\beta))&\leq S_A(r_0,\Pi_A^{-1}(\beta)).
\end{align}
Assume for a contradiction that $s>S_A(r_0,\Pi_A^{-1}(\beta))$. Let $(r_n)_{n\in\N}$ be some sequence such that $r_n\to r_0$ as $n\to\infty$ and
\begin{align*}
   \lim_{n\to\infty} S_A(r_n,\Pi_A^{-1}(\beta))&=\limsup_{r\to r_0} S_A(r,\Pi_A^{-1}(\beta)).
\end{align*}
We infer that
\begin{align*}
   S_A(r_0,\Unp(A))&=\lim_{n\to\infty} S_A(r_n,\Unp(A))\\
   &=\lim_{n\to\infty} S_A(r_n,\Pi_A^{-1}(\beta))+\lim_{n\to\infty} S_A(r_n,\Pi_A^{-1}(N(A)\setminus\beta))\\
   &=s+\lim_{n\to\infty} S_A(r_n,\Pi_A^{-1}(N(A)\setminus\beta))\\
   &>S_A(r_0,\Pi_A^{-1}(\beta))+\liminf_{r\to r_0} S_A(r,\Pi_A^{-1}(N(A)\setminus\beta))\\
   &\geq S_A(r_0,\Pi_A^{-1}(\beta))+ S_A(r_0,\Pi_A^{-1}(N(A)\setminus\beta))\\
   &=S_A(r_0,\Unp(A)),
\end{align*}
which is impossible. Hence \eqref{eq:limsup-mur} holds. Note that the second limit in the second line above exists, since the limit in the first line exists. For the last inequality we used that \eqref{eq:liminf-mur} holds with $\beta$ replaced by $N(A)\setminus\beta$.
\end{proof}

\begin{proof}
  [Proof of Theorem~\ref{thm:main}] 
  Fix some $R>r_0$. First note that, since $A$ is compact, the family of measures $M:=\{S_A(r,\cdot):r\in(0,R]\}$ is tight. Therefore, by Prokhorov's theorem, any sequence of measures in $M$ has a converging subsequence. Let $(r_i)_{i\in\N}$ be some sequence (in $(0,R]$) such that $r_i\to r_0$ as $i\to\infty$ and such that \begin{align}
     \mu_i:=S_A(r_i,\cdot)\to \mu \text{ weakly, as } i\to\infty,
  \end{align}
   for some measure $\mu$. We will show below that necessarily $\mu=\mu_0:=S_{r_0}(A,\cdot)$. Since this implies in particular that any such converging sequence $(\mu_i)$ has the same weak limit, the claimed weak convergence  $S_{r}(A,\cdot)\to S_{r_0}(A,\cdot)$, as $r\to r_0$, follows.

  Clearly, the limit measure $\mu$ satisfies $\spt\mu=\partial A_{r_0}$ (as does $\mu_0$). In order to show $\mu=\mu_0$, it therefore suffices to prove that \begin{align}
     \label{eq:claim-closed-sets} \mu(F)=\mu_0(F)
  \end{align}
 for any (relatively) closed set $F\subset \bd A_{r_0}\cap\Unp(A)$. Indeed, since the closed sets form an intersection stable generator of the Borel $\sigma$-algebra, \eqref{eq:claim-closed-sets} implies $\mu(\cdot\cap\Unp(A))= \mu_0(\cdot\cap\Unp(A))$. Moreover, since $\mu_0(\Unp(A)^c)=0$ and the total masses of the measures $\mu_i$ converge to the total mass of $\mu_0$ ($\H^{d-1}(\bd A_{r_i})\to\H^{d-1}(\bd A_{r_0})$ as $i\to\infty$), we also have $\mu(\Unp(A)^c)=0$, and hence $\mu=\mu_0$.

    It remains to verify \eqref{eq:claim-closed-sets}. Let us first consider some relatively open set $G\subset \bd A_{r_0}\cap\Unp(A)$. Since the restriction $\Pi_A|_{\bd A_{r_0}\cap \Unp(A)}$ of $\Pi_A$ to the set $A_{r_0}\cap\Unp(A)$ has a Lipschitz continuous inverse (given by $(x,n)\mapsto x+r_0 n$), its image $\Pi_A(G)\subset N(A)$ is open, and since $\Pi_A$ is continuous, the preimage $\Pi_A^{-1}(\Pi_A(G))$ is open in $\Unp(A)\setminus A$. Therefore, by the Portmanteau theorem (see e.g.~\cite{Bil99}), the weak convergence $\mu_i\to\mu$ implies that
    $$
    \liminf_{i\to\infty} \mu_i(\Pi_A^{-1}(\Pi_A(G)))\geq \mu(\Pi_A^{-1}(\Pi_A(G)))=\mu(\Pi_A^{-1}(\Pi_A(G))\cap\bd A_{r_0})=\mu(G).
    $$
    Further, since $r_0$ is a differentiability point of $V_A$, by Lemma~\ref{lem:conv-PiA}, we also have
    $$
    \lim_{i\to\infty} \mu_i(\Pi_A^{-1}(\Pi_A(G)))= \mu_0(\Pi_A^{-1}(\Pi_A(G)))=\mu_0(\Pi_A^{-1}(\Pi_A(G))\cap\bd A_{r_0})=\mu_0(G).
    $$
    Therefore, we get $\mu(G)\leq \mu_0(G)$ for any $G$ relatively open in $\bd A_{r_0}\cap\Unp(A)$. With a similar argument, we obtain $\mu(F)\geq \mu_0(F)$ for any relatively closed set $F\subset \bd A_{r_0}\cap\Unp(A)$. (Recall that, by the Portmanteau theorem, $\mu_i\to\mu$ implies $\limsup_{i\to\infty} \mu_i(C)\leq \mu(C)$ for any closed set $C\subset\R^d$.) Moreover, since any Borel measure on a metric space is outer regular, we infer that
    \begin{align*}
      \mu(F)&=\inf\{\mu(G): G\supset F \text{ is relatively open in } A_{r_0}\cap\Unp(A)\}\\
      &\leq\inf\{\mu_0(G): G\supset F \text{ is relatively open in } A_{r_0}\cap\Unp(A)\}=\mu_0(F).
    \end{align*}
    Hence, we have shown $\mu(F)=\mu_0(F)$ any relatively closed set $F\subset \bd A_{r_0}\cap\Unp(A)$ as claimed in \eqref{eq:claim-closed-sets} and this completes the proof of Theorem~\ref{thm:main}.
\end{proof}

\section{Non-differentiability points of the parallel volume} \label{s_N}

For any compact set $A\subset\R^d$ we denote by $N_A\subset(0,\infty)$ the set of all radii $s>0$ such that $r\mapsto V_A(r):=\lambda_d(A_r)$ is not differentiable at $s$. It is well known that $N_A$ is at most countable (see \cite{Stacho}). Note that $N_A$ may be empty, e.g., if $A$ is convex. Moreover, for any compact set $A\subset\R^d$, the set $N_A$ is bounded. More precisely, $N_A\subset(0,\diam(A))$ (see \cite[p.~1038]{Fu85}).
Fixing the space dimension $d\in\N$, we address the following question: \emph{For which nonempty countable sets $N\subset (0,\infty)$  does there exist some compact $A\subset\R^d$ such that $N_A=N$?}

We start with the case $d=1$. For any compact subset $A\subset\R$, the complement of $A$ consists of at most countably many open intervals, two of them unbounded. Ignoring the latter, we denote by $\ell_1, \ell_2,\ldots$ the lengths of the bounded complementary intervals of $A$ in non-increasing order. The sequence $(\ell_j)_{j\in\N}$ is called the \emph{fractal string} associated with $A$. (Observe that 
the lengths $\ell_j$ sum up to $\lambda_1(\conv(A)\setminus A)<\diam(A)<\infty$. Hence, for each $\delta>0$, there are at most finitely many elements in the sequence $(\ell_j)$ such that $\ell_j>\delta$.)
It is easy to see that $V_A$ is not differentiable at some $r>0$ if and only if $2r$ appears in the fractal string. Hence
$$
N_A=\{r\in(0,\infty): 2r=\ell_j \text{ for some } j\in\N\}.
$$
This makes clear that a countable set is possible as a set of non-differentiability points of a compact set $A\subset \R$ only if it is summable. It turns out that summability is not only necessary but also sufficient, as we stated in Theorem~\ref{thm:nondiff-points-dim1}.

\begin{proof}[Proof of Theorem~\ref{thm:nondiff-points-dim1}]
  Let $A\subset\R$ be compact. Let $(\ell_j)_{j\in\N}$ be the associated fractal string as defined above. Then clearly, $\sum_j \ell_j\leq \diam(A)<\infty$. Since
  $$
  N_A=\{r\in(0,\infty): 2r=\ell_j \text{ for some } j\in\N\}
  $$
  it follows that $\sum_{s\in N_A} s<\infty$ as claimed.

  To see the reverse implication, let $N\subset(0,\infty)$ be an arbitrary set such that $\sum_{s\in N} s<\infty$ holds. Denote by $s_1,s_2,\ldots$ the elements of $N$ in decreasing order. This may be a finite or an infinite sequence. In the finite case set $s_n:=0$ for $n>\card(N)$. Let $a_0:=0$ and $a_n:=\sum_{j=1}^n 2s_j$ for any $n\in\N$. Finally set $a_\infty:=\sum_{j=1}^\infty 2s_j$, which is finite due to the assumption $\sum_{s\in N} s<\infty$.
 Then the set
  $$
  A:=\{a_n: n\in\N_0\}\cup \{a_\infty\}
 $$
 is contained in the bounded interval $[0, a_\infty]$. Hence $A$ is bounded.
 Moreover, $A$ is either finite or it has a unique accumulation point at $a_\infty$. Hence $A$ is closed. Finally, we have $N_A=N$, since for each $s\in N$ we have a complementary interval of length $2s$ in $A$ causing a jump by $2$ at $s$ of the parallel surface area $r\mapsto \Ha^0(A_r)$ of $A$.
\end{proof}

The above one-dimensional example can be easily adapted to higher dimension, considering Cartesian products with some flat pieces. Choosing a better arrangement, one can even construct compact sets $A\subset\R^d$ with $N_A=N$ whenever $\sum_{s\in N}s^d<\infty$. We will describe such a construction below, see Proposition~\ref{prop:NC-dimd}. It shows that summability of the $d$-th powers is a sufficient condition for a set $N$ to be a set of non-differentiability points for a compact set in $\R^d$. However, for $d\geq 2$, this condition is not necessary. Using the methods of \cite{RZ20}, we derive a weaker condition which turns out to be necessary and sufficient in dimension $d=2$. We have formulated this characterization already in Theorem~\ref{thm:dim2} above, using the notion of \emph{gap sums} introduced in \cite{BT54}.
Recall that, given a compact set $K\subset\R$ and some $\alpha>0$, the \emph{degree-$\alpha$ gap sum} of $K$ is
$$G_\alpha(K):=\sum_j\ell_j^\alpha,$$
where $(\ell_j)$ is the fractal string associated with $K$. (Note that $G_\alpha(K)$ agrees with the value of the geometric zeta function of the fractal string at $\alpha$, cf.\ \cite{La06}.)

We are now ready to prove Theorem~\ref{thm:dim2} which provides a characterization of all possible sets of non-differentiability points of the volume function for planar compact sets.

\begin{proof}[Proof of Theorem~\ref{thm:dim2}]
The necessity of the condition in (i) follows directly from \cite[Theorem~1.1]{RZ20}: if $N_A=N$ for some compact set $A\subset\R^2$, then $\overline{N}\subset\cv(A)$. (Here we use the relation $N_A\subset\cv(A)$, see Remark~\ref{rem}, and the fact that $\cv(A)$ is closed.) Hence, $\overline{N}$ must be Lebesgue null and $G_{1/2}(\overline{N})<\infty$. In the same way, the necessity of the condition in (ii) follows from \cite[Theorem~1.2]{RZ20}. We will show the sufficiency by adapting some constructions from \cite{RZ20}.

To prove the sufficiency in (i), consider a countable bounded set $N\subset[\ep,\infty)$ with $\lambda_1(\overline{N})=0$ and $G_{1/2}(\overline{N})<\infty$. By \cite[Proposition~3.9]{RZ20}, for $b:=\max\overline{N}$ and $g(s):=\sqrt{2b}\, G_{1/2}(\overline{N}\cap [\ep,s])$, the set
$$F:=\{(g(s),\pm s)\in\R^2:\, s\in \overline{N}\}$$
fulfills $\overline{N}\subset\cv(F)$. We shall modify the set $F$ by adding small horizontal segments on the places of points $(g(s),\pm s)$, $s\in N$, and obtain a compact set $A$ with $N= N_A$.

To this end, we choose positive numbers $\gamma_s>0$, $s\in N$, such that $\Gamma:=\sum_{s\in N}\gamma_s<\infty$, and we set $\gamma_s:=0$ for $s\in\overline{N}\setminus N$. We denote $\Gamma_s:=\sum_{t<s}\gamma_t$ and $J_s:=[g(s)+\Gamma_s,g(s)+\Gamma_s+\gamma_s]$, $s\in\overline{N}$.
Note that, due to the monotonicity of $g$, the family $\{J_s:\, s\in\overline{N}\}$ consists of pairwise disjoint compact intervals. The length of $J_s$ is $\gamma_s$, in particular, $J_s$ is degenerate if $s\in\overline{N}\setminus N$. Further,
$$\sup_{s\in\overline{N}}\max J_s=\sqrt{2b}\,G_{1/2}(\overline{N})+\Gamma.$$
We set
\begin{align*}
B:=&\bigcup_{s\in\overline{N}}\bigcup_{x\in J_s}B((x,0),s),\\
R:=&[-b,b+\sqrt{2b}\,\cG_{1/2}(\overline{N})+\Gamma]\times[-b,b]\quad\text{ and}\\
A:=&R\setminus\INt B.
\end{align*}
Note that $R$ is a rectangle containing $B$, and $\INt B$ (a union of open balls with centres on the $x$-axis) thus forms the bounded part of the complement of $A$. Due to the symmetry, the only critical points of $A$ can lie on the $x$-axis.

Let $x\in J_s$ for some $s\in\overline{N}$. Since $\INt B((x,0),s)\cap A=\emptyset$, we have $\dist((x,0),A)\geq s$. We shall show that both, $(x,-s)$ and $(x,s)$, belong to $A$, which will imply that $(x,0)\in\crit(A)$ and $d((x,0),A)=s$. It suffices to show that $\|(x',0)-(x,\pm s)\|\geq s'$ whenever $s'\neq s$ and $x'\in J_{s'}$. Indeed, this will imply that the points $(x,\pm s)$ are not contained in the ball $B((x',0),s')$. If $x'<x$, then, by definition, $s'<s$ and the inequality is obvious. If $x'>x$, we have
$$
s'-s=\sum_{I\in\cG(\overline{N}\cap[s,s'])}|I|\leq\left(\sum_{I\in\cG(\overline{N}\cap[s,s'])}\sqrt{|I|}\right)^2=(2b)^{-1}(g(s')-g(s))^2.
$$
Here we used the notation $\cG(K)$ for the set of all maximal bounded open intervals of the complement of a compact set $K\subset\R$. Thus,
$$s'^2-s^2=(s'+s)(s'-s)\leq 2b(s'-s)\leq (g(s')-g(s))^2\leq (x'-x)^2,$$
which implies that
$$\|(x,\pm s)-(x',0)\|^2=(x-x')^2+s^2\geq s'^2.$$
It remains to consider possible critical points $(x,0)$ for $x\not\in \bigcup_sJ_s$. If $x\not\in\conv(\bigcup_sJ_s)$, then $(x,0)$ cannot be critical for $A$. If $x\in\conv(\bigcup_sJ_s)$, then $x$ has nearest left and right neighbours in $\bigcup_sJ_s$, i.e., points $x_1<x<x_2$, $x_1\in J_{s_1}$, $x_2\in J_{s_2}$ for some $s_1<s_2$, $s_1,s_2\in\overline{N}$ such that the open interval $(x_1,x_2)$ is disjoint from $\bigcup_sJ_s$. Consider the intersection $\partial B((x_1,0),s_1)\cap \partial B((x_2,0),s_2)$. If it consists of two points $(x',\pm y')$, then these are the only closest points from $(x,0)$ to $A$, and hence $(x,0)$ is critical only if $x=x'$. This is an isolated critical point of $A$, hence it cannot contribute to $N_A$. If the intersection $\partial B((x_1,0),s_1)\cap \partial B((x_2,0),s_2)$ is empty or a singleton, then there is some $x'\in\R$ such that $(x',0)$ is the unique nearest point from $(x,0)$ to $A$. Hence $x$ is not critical for $A$.

Now we will prove the sufficiency of \eqref{int_bound} in (ii); we will use an analogous procedure as in the proof of \cite[Proposition~3.16]{RZ20}.
For this purpose, let $N\subset(0,\infty)$ be a bounded, countable set satisfying $\lambda_1(\overline{N})=0$ and condition \eqref{int_bound}.
As before, let $b:=\max\overline{N}$ and denote $\delta_n:=b2^{-n}$, $n\in\N$. Applying \cite[Lemma~3.15]{RZ20}, we obtain for each $n\in\N$ a decomposition
$$K_n:=\overline{N}\cap[\delta_{n+1},\delta_n]=K_{n,1}\cup\dots\cup K_{n,p_n}$$
with compact sets $K_{n,i}$ such that $p_n\leq \delta_n^{-1/2}G_{1/2}(K_n)+1$ and
$$G_{1/2}(K_{n,i})\leq 2\delta_n^{1/2},\quad i=1,\dots,p_n.$$
Applying the construction from the first part of the proof to each of the sets $K_{n,i}\cap N$, we obtain compact sets $A_{n,i}\subset\R^2$ contained in rectangles $R_{n,i}$ such that
$K_{n,i}\cap N= N_{A_{n,i}}$
and
$$R_{n,i}= [-\delta_n,\delta_n+\sqrt{2\delta_n}\,G_{1/2}(K_{n,i})+\Gamma_{n,i}]\times[-\delta_n,\delta_n],
$$
for some parameters $\Gamma_{n,i}>0$, which we can choose freely. Choosing $\Gamma_{n,i}:=(2-\sqrt{2})\sqrt{\delta_n}\,G_{1/2}(K_{n,i})$, we obtain
\begin{align*}
\sum_{n=1}^\infty\sum_{i=1}^{p_n}\lambda_2(R_{n,i})&=\sum_{n=1}^\infty\sum_{i=1}^{p_n} 4(\delta_n^2+\delta_n^{3/2}G_{1/2}(K_{n,i}))\\
&\leq 12\sum_{n=1}^\infty\delta_n^2p_n\\
&\leq 12\sum_{n=1}^\infty\delta_n^2(\delta_n^{-1/2}G_{1/2}(K_n)+1)\\
&= 12\sum_{n=1}^\infty\delta_n^{3/2}G_{1/2}(K_n)+12\sum_{n=1}^\infty\delta_n^2.
\end{align*}
The last sum is clearly finite. We further use that \eqref{int_bound} implies the finiteness of the first sum, that is, $\sum_{n=1}^\infty\delta_n^{3/2}G_{1/2}(K_n)<\infty$. (In fact, these two conditions are equivalent, see \cite[p.~315]{RZ20}.) Thus, the sum of the areas of all the rectangles $R_{n,i}$ is finite. Since the diameters of the sets $R_{n,i}$ are uniformly bounded, according to the solution of the ``\emph{potato-sack problem}'' (see \cite{ABMU81}), we can find a \emph{bounded packing} of the sets $R_{n,i}$, i.e., there exist shifts $a_{n,i}\in\R^2$ such that $R_{n,i}+a_{n,i}$ are pairwise disjoint and $\bigcup_{n=1}^\infty\bigcup_{i=1}^{p_n}(R_{n,i}+a_{n,i})\subset B(0,R)$ for some $R>0$. We define
$$A:=\bigcup_{n=1}^\infty\bigcup_{i=1}^{p_n}(A_{n,i}+a_{n,i})\cup \left(B(0,R)\setminus\bigcup_{n=1}^\infty\bigcup_{i=1}^{p_n} (R_{n,i}+a_{n,i})\right).$$
It is easy to see that $A$ is a compact set and its complement has the following (disjoint) connected components: the complement of $B(0,R)$ and the bounded connected components of $(A_{n,i}+a_{n,i})^C$, $i=1,\dots,p_n$, $n\in\N$. Thus,
$$N_A=\bigcup_{n=1}^\infty\bigcup_{i=1}^{p_n}N_{A_{n,i}}=N,$$
completing the proof of part (ii) of Theorem~\ref{thm:dim2}.
\end{proof}

\begin{rem}  \label{rem_s^2}
If $N=\{b_i:\, i=1,2\dots\}$ for a monotone sequence $b_1>b_2>\ldots$ such that $\sum_ib_i^2<\infty$, then condition \eqref{int_bound} in Theorem~\ref{thm:dim2} is satisfied for $N$. Indeed, \eqref{int_bound} can be rewritten in this case as
$$I:=\int_0^{b_2}\sum_{j:\, b_{j+1}\geq r}\sqrt{b_j-b_{j+1}}\, \sqrt{r}\, dr<\infty,$$
and since
\begin{align*}
I&\leq\sum_{i=1}^\infty (b_{i+1}-b_{i+2})\sum_{j=1}^i\sqrt{b_j-b_{j+1}}\, \sqrt{b_i}\\
&=\sum_{j=1}^\infty \sqrt{b_j-b_{j+1}}\sum_{i=j}^\infty \sqrt{b_i}(b_{i+1}-b_{i+2})\\
&\leq\sum_{j=1}^\infty \sqrt{b_j-b_{j+1}}\sqrt{b_j}\sum_{i=j}^\infty (b_{i+1}-b_{i+2})\\
&\leq\sum_{j=1}^\infty \sqrt{b_j-b_{j+1}}\sqrt{b_j}\,b_{j+1},
\end{align*}
the H\"older inequality implies
$$I\leq\sqrt{\sum_{j=1}^\infty(b_j-b_{j+1})b_j}\sqrt{\sum_{j=2}^\infty b_j^2}<\infty,$$
showing that the latter condition is satisfied. Hence square-summability of an infinite set $N$ implies condition \eqref{int_bound} and $\lambda_1(\overline{N})=0$. (And clearly any finite set $N$ satisfies the conditions $G_{1/2}(\overline{N})<\infty$ and $\lambda_1(\overline{N})=0$ in part (i) of Theorem~\ref{thm:dim2}.) But the conditions in Theorem~\ref{thm:dim2} are strictly weaker. Indeed, the gap-sum condition of Theorem~\ref{thm:dim2}~(i) is translation invariant, hence, the set $N$ can have an accumulation point greater than $0$, implying $\sum_{s\in N}s^2=\infty$.
\end{rem}

We provide some examples of valid sets of non-differentiability points (of the volume function of some compact subset of $\R^2$) showing that such sets may have a rich structure. In particular, they may have any Minkowski dimension less than $4/5$.
\begin{ex}
  \label{ex:cantor} For $q\in(0,1/2)$ let $F_q\subset\R$ be the self-similar set generated by the two mappings $x\mapsto qx$ and $x\mapsto qx +(1-q)$. (The set can be obtained iteratively from $[0,1]$ by removing in the $n$-th step ($n\in\N$) from each of the intervals remaining from the previous step (each of length $q^{n-1}$) a centred open interval of length $(1-2q)q^{n-1}$.) It is well known that $F_q$ has Minkowski dimension $\dim_M F_q={\log 2}/{\log (1/q)}<1$.
  Note that $[0,1]\setminus F_q$ has countably many connected components (open intervals). Let $E_q$ be the set of all endpoints of these intervals. Observe that among these intervals there are exactly $2^{k}$  of length $(1-2q)q^k$ for each $k\in\N_0$. Note also that $\overline{E}_q=F_q$. Hence, for any $\alpha>0$,
  $$
  G_\alpha(\overline{E}_q)=G_\alpha(F_q)=\sum_{k=0}^\infty 2^k (1-2q)^\alpha q^{k\alpha}=(1-2q)^\alpha \sum_{k=0}^\infty (2 q^\alpha)^k.
  $$
  For $\alpha=1/2$, the gap sum will be finite if and only if $q\in(0,1/4)$.
  Now define for some $\eps>0$ and $q\in(0,1/4)$ the set $N_q:=\eps+E_q$. Then, clearly, $N_q$ satisfies the conditions of Theorem~\ref{thm:dim2}(i) and therefore it is a valid set non-differentiability points.
  According to part (ii), also the set $E_q$ itself is valid for each $q\in(0,1/4)$, since condition \eqref{int_bound} is satisfied for $N=E_q$. This shows that sets of non-differentiability points may have any Minkowski dimension between $0$ and $1/2$.

  In contrast, the sets $E_q$ for $q\in[1/4,1/2)$ cannot occur as sets of non-differentiabi\-lity points, since due to the self-similarity one has in this case, for any $r<1-q$,
  $$
  G_{1/2}(F_q\cap [r,\infty))\geq G_{1/2}(F_q\cap [1-q,\infty))=\sqrt{q}G_{1/2}(F_q)=\infty.
  $$
  Thus condition \eqref{int_bound} is not satisfied.

  However, as indicated after Theorem~\ref{thm:dim2}, it is possible for a valid set $N$ of non-differentiability points to have a Minkowski dimension larger than $1/2$. But then the Minkowski dimension of $N$ must be determined by its behaviour near $0$ while away from $0$ it should not exceed $1/2$ in the following sense: for each $\eps>0$ the set $N\cap[\eps,\infty)$ must satisfy $G_{1/2}(N\cap[\eps,\infty))<\infty$ implying in particular $\dim_M(N\cap[\eps,\infty))\leq\frac 12$. We also provide examples of such sets. For $q\in[1/4,1/2)$, we rearrange the intervals of $[0,1]\setminus F_q$ in the following way. We place them side by side in $[0,1]$ with no space between them in non-increasing order of their lengths starting with the largest interval of length $(1-2q)$ placed such that 1 is one of its endpoints. The resulting set of endpoints is $E_q'=\{1,2q,2q^2+q,4q^2,\ldots\}$. It clearly satisfies $\overline{E_q'}=E_q'\cup\{0\}$, i.e., it only accumulates at $0$. Thus $E_q'\cap[r,\infty)$ is finite and therefore $G_{1/2}(\overline{E}_q'\cap[r,\infty))<\infty$ for each $r>0$. More precisely, we have for $q\in(1/4,1/2)$ and $r\in[(2q)^n,(2q)^{n-1})$, $n\in\N$,
  \begin{align*}
     G_{1/2}(\overline{E}_q'\cap[r,\infty))&\leq G_{1/2}(\overline{E}_q'\cap[(2q)^n,\infty))=\sum_{k=0}^{n-1} 2^k ((1-2q)q^{k})^{1/2}\\
     &=(1-2q)^{1/2}\sum_{k=0}^{n-1} (2q^{1/2})^k=\frac{\sqrt{1-2q}}{2\sqrt{q}-1}((2\sqrt{q})^n-1)\leq c_q (2\sqrt{q})^n,
  \end{align*}
  where $c_q:=\frac{\sqrt{1-2q}}{2\sqrt{q}-1}>0$, since $q> 1/4$. This yields
  \begin{align*}
    \int_0^\infty G_{1/2}(\overline{E_q'}\cap [r,\infty))\, \sqrt{r}\, dr
    &=\sum_{n=1}^\infty \int_{(2q)^n}^{(2q)^{n-1}} G_{1/2}(\overline{E_q'}\cap [r,\infty))\, \sqrt{r}\, dr\\
    &\leq\sum_{n=1}^\infty \int_{(2q)^n}^{(2q)^{n-1}} G_{1/2}(\overline{E_q'}\cap [(2q)^n,\infty))\, \sqrt{(2q)^{n-1}}\, dr\\
    &\leq\sum_{n=1}^\infty c_q (2\sqrt{q})^n(\sqrt{2q})^{n-1} \int_{(2q)^n}^{(2q)^{n-1}} dr\\
    &\leq c_q \frac{1-2q}{\sqrt{2q} 2q}\sum_{n=1}^\infty (4\sqrt{2}q^2)^{n}.
  \end{align*}
  The last sum is finite if and only if $q<(2\sqrt[4]{2})^{-1}$. Therefore, condition \eqref{int_bound} in Theorem~\ref{thm:dim2} is satisfied for $N=E_q'$ for each $q\in(1/4,1/(2\sqrt[4]{2}))$, showing that $E_q'$ is a valid set of non-differentiability points for these $q$. (A similar but slightly different computation shows that the same is true for the case $q=1/4$.) Observe that the choice $q=(2\sqrt[4]{2})^{-1}$ yields a set $E_q'$ of Minkowski dimension $\dim_M E_q'=4/5$. Hence a set $N_A$ of non-differentiability points may have any Minkowski dimension less than $4/5$. Finally, in \cite[Theorem~3.17]{RZ20}, a compact set $F\subset\R$ is described, which satisfies \eqref{int_bound} and $\dim_M F=4/5$. Hence, Minkowski dimension $4/5$ is also possible for a set of non-differentiability points.
\end{ex}

In dimension $d>2$ we do not know a characterization of non-differentiability points of the volume function. We only provide a sufficient condition which is analogous to the summability in Theorem~\ref{thm:nondiff-points-dim1} and Remark~\ref{rem_s^2}:

\begin{prop} \label{prop:NC-dimd}
  Let $d\in\N$ and $N\subset(0,\infty)$ be such that $\sum_{s\in N} s^d<\infty$.
  Then, there exists some compact set $C\subset\R^{d}$ such that $N_C=N$.
\end{prop}

\begin{proof}
To prove this statement, we will explicitly construct a compact set $A\subset\R^d$ with $N_A=N$.
For any $s\in N$, consider the box $D_s\subset\R^d$ and its boundary $R_s$ defined by
$$D_s:= [0,3s]\times\dots\times[0,3s]\times[0,2s]\quad \text{and} \quad R_s:=\partial D_s,
$$
respectively. It is easy to see that $s$ is the only non-differentiability point of the volume function of $R_s$. Indeed, the set of critical values of $R_s$ is the $(d-1)$-dimensional cube $C_s$ lying parallel in the middle between the largest facets of $R_s$:
$$C_s=[s,2s]\times\dots\times[s,2s]\times\{s\}.$$
Note that the sum of volumes of $D_s$ converges by our assumption, and their diameters are uniformly bounded. Hence, using again the solution of the potato-sack problem (see \cite{ABMU81}), there exists a bounded packing of the boxes $D_s$, i.e., there exist points $a_s\in\R^d$ such that $(a_s+D_s)\cap (a_t+D_t)=\emptyset$, $s\neq t$, $s,t\in N$, and $\bigcup_{s\in N}(a_s+D_s)\subset B(0,R)$ for some $R$. Thus, the set
$$A:=B(0,R)\setminus\bigcup_{s\in N}\INt(a_s+D_s)$$
fulfills $N_A=N$.
\end{proof}


\begin{thebibliography}{99}

\bibitem{ABMU81}
H. Auerbach, S. Banach, S. Mazur, and S. Ulam: Problem 10.1, in: \emph{The Scottish Book}, ed. by R.D. Mauldin, Birkh\"auser, Boston, 1981, p. 74.

\bibitem{BT54}
A.S. Besicovitch,  S.J. Taylor: On the complementary intervals of a linear closed set of zero Lebesgue measure. \emph{J. London Math.\ Soc.} {\bf 29} (1954), 449--459.

\bibitem{Bil99} P.~Billingsley: {\em Convergence of Probability Measures.}  John Wiley \& Sons, Inc., 1999.

\bibitem{CLV21}
A. Chambolle, L. Lussardi, E. Villa: Anisotropic tubular neighborhoods of sets. {\em Math. Z.} {\bf 299} (2021), 1257--1274. 

\bibitem{Fe69}
   H.~Federer:  {\em Geometric Measure Theory}. Springer, 1969.

\bibitem{Fu85} J.H.G.~Fu: Tubular neighborhoods in Euclidean spaces. {\em Duke Math. J.} {\bf 52} (1985), 1025--1046.

\bibitem{HLW04} D. Hug, G. Last, W. Weil: A local
Steiner-type formula for general closed sets and applications. {\it Math. Z.} {\bf 246} (2004), 237--272.

\bibitem{HS22} D. Hug, M. Santilli: Curvature measures and soap bubbles beyond convexity. \emph{Adv. Math.} {\bf 411} (2022), 108802. 

\bibitem{La06} M.L. Lapidus, M. van Frankenhuijsen: {\em Fractal Geometry, Complex Dimensions and Zeta Functions: Geometry and spectra of fractal strings}. Springer, New York, 2006.


\bibitem{RW10} J. Rataj, S. Winter: On volume and surface area of parallel sets. {\it Indiana Univ. Math. J.} {\bf 59} (2010), 1661--1686.

\bibitem{RWZ22} J. Rataj, S. Winter, M. Z\"ahle: Mean Lipschitz-Killing curvatures for homogeneous random fractals. \emph{J. Fractal Geom.} (to appear), arXiv:2107.14431

\bibitem{RZ20}
J. Rataj, L. Zaj\'\i\v cek: Smallness of the set of critical values of distance functions in two-dimensional Euclidean and Riemannian spaces. \emph{Mathematika} {\bf 66} (2020), 297--324.

\bibitem{Stacho} L.L. Stach{\'o}: On the volume function of parallel sets. {\em Acta Sci. Math.} {\bf 38} (1976), 365--374.

\bibitem{Wi19} S. Winter: Localization results for Minkowski contents. {\it J. London Math. Soc.} {\bf 99} (2019), 553--582. 

\bibitem{Za11} M. Z\"{a}hle: Lipschitz-Killing curvatures of self-similar random fractals. {\it Trans. Amer. Math. Soc.} {\bf 363}, (2011), 2663--2684.


\end{thebibliography}
\end{document}